\documentclass[10pt,english]{article}
\usepackage{lmodern}
\usepackage[T1]{fontenc}
\usepackage[latin9]{inputenc}
\usepackage[a4paper]{geometry}
\usepackage{amsmath}
\usepackage{amsthm}
\usepackage{amssymb}
\usepackage{graphicx}

\makeatletter
\newcommand{\lyxaddress}[1]{
	\par {\raggedright #1
	\vspace{1.4em}
	\noindent\par}
}
\theoremstyle{plain}
\newtheorem{thm}{\protect\theoremname}
\theoremstyle{definition}
\newtheorem{defn}[thm]{\protect\definitionname}
\theoremstyle{remark}
\newtheorem{rem}[thm]{\protect\remarkname}
\theoremstyle{plain}
\newtheorem{cor}[thm]{\protect\corollaryname}
\theoremstyle{plain}
\newtheorem{lem}[thm]{\protect\lemmaname}

\@ifundefined{date}{}{\date{}}
\makeatother

\usepackage{babel}
\providecommand{\corollaryname}{Corollary}
\providecommand{\definitionname}{Definition}
\providecommand{\lemmaname}{Lemma}
\providecommand{\remarkname}{Remark}
\providecommand{\theoremname}{Theorem}

\begin{document}
\title{On the Calculation of Differential Parametrizations for the Feedforward
Control of an Euler-Bernoulli Beam}
\author{Bernd Kolar\thanks{E-mail: bernd.kolar@jku.at}\ , Nicole Gehring
and Markus Schöberl}
\maketitle

\lyxaddress{\begin{center}
Institute of Automatic Control and Control Systems Technology, Johannes
Kepler University Linz, Altenbergerstraße 66, 4040 Linz, Austria
\par\end{center}}
\begin{abstract}
This contribution is concerned with the motion planning for underactuated
Euler-Bernoulli beams. The design of the feedforward control is based
on a differential parametrization of the beam, where all system variables
are expressed in terms of a free time function and its infinitely
many derivatives. In the paper, we derive an advantageous representation
of the set of all formal differential parametrizations of the beam.
Based on this representation, we identify a well-known parametrization,
for the first time without the use of operational calculus. This parametrization
is a flat one, as the corresponding series representations of the
system variables converge. Furthermore, we discuss a formal differential
parametrization where the free time function allows a physical interpretation
as the bending moment at the unactuated boundary. Even though the
corresponding series do not converge, a numerical simulation using
a least term summation illustrates the usefulness of this formal differential
parametrization for the motion planning.
\end{abstract}

\section{Introduction}

The feedforward control for boundary actuated Euler-Bernoulli beams
is usually concerned with the question of how an input has to be chosen
in order to transition the beam from one steady state to another.
The flatness-based approach, originally introduced for lumped-parameter
systems, has proven very valuable for the motion planning in distributed-parameter
systems (e.g. \cite{FliessMounierRouchonRudolph_CDC1995}, \cite{LynchRudolph_IJC2002},
\cite{Rudolph_Shaker2003}, \cite{KnueppelWoittennek_TAC2015}). It
relies on a differential parametrization of the spatially-dependent
system solution in terms of a parametrizing (boundary) output, and,
in the case of the Euler-Bernoulli beam or parabolic systems like
the heat equation, involves infinite series. As these series comprise
infinitely many time derivatives of the parametrizing output, in order
to transition the system between two steady states and have the series
converge, the desired trajectory of the parametrizing output has to
be a non-analytic, smooth function of appropriate Gevrey class (e.g.
\cite{LarocheMartin_IJRNC2000}). If a Gevrey class ensuring convergence
exists, the parametrizing output is called a flat output and the differential
parametrization is a flat one (e.g. \cite{LynchRudolph_IJC2002}).
On the other hand, if no such Gevrey class exists, the differential
parametrization is specified by the prefix formal. However, even in
this case, a flatness-based motion planning may still be possible
(e.g. \cite{LarocheMartin_IJRNC2000}, \cite{WagnerMeurerZeitz2004}).

The main challenge of the flatness-based approach lies in finding
a differential parametrization. For the heat conduction problem this
is either based on the ansatz of a power series in the spatial variable
(e.g. \cite{LarocheMartin_IJRNC2000}), the ansatz in \cite{LarocheMartin_MTNS2000}
which generalizes the Brunovsky decomposition, or some kind of operational
calculus (e.g. \cite{ECC97MiniCourse}, \cite{Rudolph_Shaker2003}).
All these ideas could be applied in a straightforward way to the fully
actuated Euler-Bernoulli beam, i.e. configurations with two boundary
inputs at one end. In contrast, in the more difficult case with only
one input, considered here, all differential parametrizations found
in the literature solely rely on operational calculus (e.g. \cite{FliessMounierRouchonRudolph_ESAIM1997},
\cite{Rudolph_Shaker2003}, \cite{RudolphWoittennek_Bookchapter2003},
\cite{MeurerSchroeckKugi_CDC2010}), to the best of the authors' knowledge.

In this paper, we demonstrate that a time-domain approach similar
to \cite{LarocheMartin_MTNS2000} is also applicable to an Euler-Bernoulli
beam with one boundary input. One of our main results is an advantageous
representation of the set of all formal differential parametrizations
of the beam. Based on this representation, it is easy to identify
a distinguished differential parametrization that is a flat one and
well-known from the literature (see e.g. \cite{Rudolph_Shaker2003}).
As no physical interpretation is known for this flat output, we discuss
an interesting formal differential parametrization with a parametrizing
output that allows an interpretation as the bending moment at the
unactuated boundary. Although the corresponding series do not converge
(for any non-analytic function), inspired by the computations with
divergent series in \cite{LarocheMartin_IJRNC2000}, simulation results
based on a least term summation illustrate that this formal differential
parametrization can still be useful for motion planning.

\paragraph*{Notation}

We use the convention that the set of natural numbers $\mathbb{N}$
includes $0$. Thus, $(a_{k})_{k\in\mathbb{N}}$ denotes the infinite
sequence $(a_{0},a_{1},a_{2},\ldots)$. By $(a_{0},a_{1},\ldots,a_{k})$
we denote a finite sequence with the last element $a_{k}$ for some
fixed $k$.

\section{\label{sec:beam} A Series Ansatz for Solutions of the Euler-Bernoulli
Beam}

In this contribution, we consider an Euler-Bernoulli beam with a clamped
end at $z=0$ and a free end at $z=1$, with the bending moment serving
as the control input $u(t)$ at the latter. The deflection $w(z,t)$
satisfies the (normalized) beam equation \begin{subequations}\label{eq:PDE_RB_Beam}
\begin{equation}
\partial_{t}^{2}w(z,t)=-\partial_{z}^{4}w(z,t)\,,\quad0\leq z\leq1\,,\,t\geq0\label{eq:PDE_Beam}
\end{equation}
and the boundary conditions
\begin{equation}
\begin{array}{rclcrcl}
w(0,t) & = & 0\,, & \quad & \partial_{z}^{2}w(1,t) & = & u(t)\\
\partial_{z}w(0,t) & = & 0\,, &  & \partial_{z}^{3}w(1,t) & = & 0\,.
\end{array}\label{eq:BC_Beam}
\end{equation}
\end{subequations} This is a classical beam configuration that has
been studied in the context of flatness-based feedforward control,
e.g., in \cite{Rudolph_Shaker2003} and \cite{HaasRudolph1999}.\footnote{The general constant-coefficient case
\[
\mu\partial_{t}^{2}w(z,t)=-EI\partial_{z}^{4}w(z,t)\,,\quad0\leq z\leq L\,,\,t\geq0
\]
with the linear mass density $\mu>0$, the flexural rigidity $EI>0$,
and a spatial domain $[0,L]$ can always be traced back to the normalized
case (\ref{eq:PDE_RB_Beam}) by transformations of the independent
variables $z$ and $t$.}

The objective of the present paper is the construction of formal differential
parametrizations for the Euler-Bernoulli beam (\ref{eq:PDE_RB_Beam}).
We call a representation\begin{subequations}\label{eq:Param_Ansatz}
\begin{eqnarray}
w(z,t) & = & \sum_{k=0}^{\infty}\alpha_{k}(z)y^{(k)}(t)\label{eq:Param_w_Ansatz}\\
u(t) & = & \sum_{k=0}^{\infty}\beta_{k}y^{(k)}(t)\label{eq:Param_u_Ansatz}
\end{eqnarray}
\end{subequations} of the deflection $w(z,t)$ and the input $u(t)$
by a parametrizing output $y(t)$ and its time derivatives a formal
differential parametrization of the system (\ref{eq:PDE_RB_Beam}),
if the series (\ref{eq:Param_w_Ansatz}) and (\ref{eq:Param_u_Ansatz})
formally satisfy the PDE (\ref{eq:PDE_Beam}) and the boundary conditions
(\ref{eq:BC_Beam}) for arbitrary smooth functions $y(t)$. It is
important to emphasize the word formal, since we are dealing with
formal solutions\footnote{The series (\ref{eq:Param_w_Ansatz}) and (\ref{eq:Param_u_Ansatz})
are formal solutions if they satisfy (\ref{eq:PDE_Beam}) and (\ref{eq:BC_Beam})
after formally interchanging differentiation and summation, even if
they do not converge.} and discuss the problem of finding parametrizations (\ref{eq:Param_Ansatz})
separately from the question of convergence, see also \cite{WagnerMeurerZeitz2004}
or \cite{LarocheMartin_IJRNC2000}. For a convergence analysis, $y(t)$
has to be restricted to so-called Gevrey functions with appropriately
bounded derivatives, see e.g. \cite{LarocheMartin_IJRNC2000}.
\begin{defn}
A smooth function $y(t)$ defined on $\mathbb{R}^{+}$ is Gevrey of
order $\gamma$ if there exist constants $M,R>0$ such that
\begin{equation}
\sup_{t\in\mathbb{R}^{+}}\left|y^{(m)}(t)\right|\leq M\tfrac{(m!)^{\gamma}}{R^{m}}\,,\quad\forall m\in\mathbb{N}\,.\label{eq:Gevrey_class}
\end{equation}
\end{defn}

Gevrey functions of order $\gamma=1$ are analytic. For planning transitions
between equilibria of the system (\ref{eq:PDE_RB_Beam}), which are
characterized by constant values of $y(t)$, they cannot be used.
The reason is that any analytic function that is constant on an open
subset of $\mathbb{R}^{+}$ is constant everywhere. In contrast, this
is no longer true for Gevrey functions $y(t)$ of order $\gamma>1$.
If the convergence of the series (\ref{eq:Param_w_Ansatz}) and (\ref{eq:Param_u_Ansatz})
can be shown for a Gevrey class with $\gamma>1$, the parametrizing
output $y(t)$ is called a flat output.
\begin{rem}
In the context of motion planning, the question arises whether the
span over the set $\{\alpha_{k}(z)$, $k\in\mathbb{N}\}$ of functions
of a given parametrization (\ref{eq:Param_w_Ansatz}) is dense in
the state space, see also \cite{LarocheMartin_MTNS2000}. However,
this question is not within the scope of the present paper. 
\end{rem}

In order to derive conditions on the functions $\alpha_{k}(z)$ and
coefficients $\beta_{k}$, let us plug the ansatz (\ref{eq:Param_Ansatz})
into the PDE (\ref{eq:PDE_Beam}) and the boundary conditions (\ref{eq:BC_Beam}).
Since the resulting equations must hold for arbitrary smooth functions
$y(t)$, after formally interchanging differentiation and summation,
the factors of all time derivatives of $y(t)$ have to vanish. Hence\begin{subequations}\label{eq:Conditions_alpha}
\begin{eqnarray}
\alpha_{0}^{\prime\prime\prime\prime}(z) & = & 0\nonumber \\
\alpha_{1}^{\prime\prime\prime\prime}(z) & = & 0\nonumber \\
\alpha_{k}^{\prime\prime\prime\prime}(z) & = & -\alpha_{k-2}(z)\,,\quad k\geq2\label{eq:ODE_alpha}
\end{eqnarray}
and
\begin{equation}
\left.\begin{array}{cclcccl}
\alpha_{k}(0) & = & 0\,, & \quad & \alpha_{k}^{\prime\prime}(1) & = & \beta_{k}\\
\alpha_{k}^{\prime}(0) & = & 0\,, &  & \alpha_{k}^{\prime\prime\prime}(1) & = & 0
\end{array}\quad\right\} \quad k\geq0\,.\label{eq:BC_alpha}
\end{equation}
\end{subequations}Therein, $^{\prime}$ denotes the differentiation
with respect to the spatial variable $z$. With (\ref{eq:Conditions_alpha})
we have to solve a sequence of boundary value problems in the independent
variable $z$. Obviously, for all $k\geq2$, the function $\alpha_{k}(z)$
follows from a fourfold integration of $-\alpha_{k-2}(z)$. The integration
constants are determined by the boundary values (\ref{eq:BC_alpha}).
Consequently, the functions $\alpha_{k}(z)$ of the series representation
(\ref{eq:Param_w_Ansatz}) of the deflection $w(z,t)$ are uniquely
determined by the coefficients $(\beta_{0},\beta_{1},\ldots,\beta_{k})$
of the series representation (\ref{eq:Param_u_Ansatz}) of the input
$u(t)$. More precisely, the functions $\alpha_{k}(z)$ with even
$k$ are determined by the coefficients $(\beta_{0},\beta_{2},\ldots,\beta_{k})$
with even indices, the functions $\alpha_{k}(z)$ with odd $k$ by
the coefficients $(\beta_{1},\beta_{3},\ldots,\beta_{k})$ with odd
indices. Hence, the functions $\alpha_{k}(z)$ with even and odd indices
can be calculated independently. In particular, setting all $\beta_{k}$
with odd $k$ to zero yields a parametrization of the form\begin{subequations}\label{eq:Param_even}
\begin{eqnarray}
w(z,t) & = & \sum_{k=0}^{\infty}\alpha_{2k}(z)y^{(2k)}(t)\label{eq:Param_w_even}\\
u(t) & = & \sum_{k=0}^{\infty}\beta_{2k}y^{(2k)}(t)\,,\label{eq:Param_u_even}
\end{eqnarray}
\end{subequations}where only time derivatives of even order occur.
In the remainder of the paper, for simplicity, we restrict ourselves
to this special case. The calculations for the case with odd indices
could be performed in an analogous way.

Integrating $\alpha_{0}^{\prime\prime\prime\prime}(z)=0$ four times
and taking account of (\ref{eq:BC_alpha}) yields the first function
\begin{equation}
\alpha_{0}(z)=\tfrac{1}{2}\beta_{0}z^{2}\label{eq:alpha_0}
\end{equation}
of the series (\ref{eq:Param_w_even}). Note that $\alpha_{0}(z)$
corresponds to an equilibrium profile of the beam, since all time
derivatives of $y(t)$ vanish in steady state. Continuing by solving
$\alpha_{2}^{\prime\prime\prime\prime}(z)=-\alpha_{0}(z)$ and $\alpha_{4}^{\prime\prime\prime\prime}(z)=-\alpha_{2}(z)$,
it can be shown that the functions $\alpha_{2k}(z)$ are polynomials
of the form
\begin{equation}
\alpha_{2k}(z)=\sum_{i=0}^{k}\left(c_{k,i}z^{4i+2}+d_{k,i}z^{4i+3}\right)\label{eq:Ansatz_alpha2k}
\end{equation}
with $d_{k,k}=0$, where $c_{k,i}$ denotes the coefficients of the
even powers $z^{4i+2}$, and $d_{k,i}$ denotes the coefficients of
the odd powers $z^{4i+3}$.

In principle, by successive integration we could calculate all functions
$\alpha_{2k}(z)$ -- up to some arbitrary index $k$ -- in terms
of the sequence $(\beta_{0},\beta_{2},\ldots,\beta_{2k})$. However,
this is cumbersome and does not answer the important question how
one has to choose the sequence $(\beta_{2k})_{k\in\mathbb{N}}$ in
order for the series (\ref{eq:Param_w_even}) and (\ref{eq:Param_u_even})
to converge for a sufficiently large class of functions $y(t)$. Therefore,
in the following section, we study the connections between the sequence
$(\beta_{2k})_{k\in\mathbb{N}}$ and the coefficients of the polynomials
(\ref{eq:Ansatz_alpha2k}) in more detail.

\section{\label{sec:FormalDiffParam}Derivation of Formal Differential Parametrizations}

First, we determine how the coefficients of (\ref{eq:Ansatz_alpha2k})
depend on the coefficients of
\begin{equation}
\alpha_{2(k-1)}(z)=\sum_{i=0}^{k-1}\left(c_{k-1,i}z^{4i+2}+d_{k-1,i}z^{4i+3}\right)\,.\label{eq:Ansatz_alpha2k-1}
\end{equation}
Integrating $-\alpha_{2(k-1)}(z)$ four times and taking account of
(\ref{eq:BC_alpha}), a comparison with (\ref{eq:Ansatz_alpha2k})
shows that the coefficients of (\ref{eq:Ansatz_alpha2k}) can be calculated
from the coefficients of (\ref{eq:Ansatz_alpha2k-1}) according to
\begin{eqnarray}
c_{k,0} & = & \tfrac{1}{2}\beta_{2k}-\tfrac{1}{2}\sum_{i=0}^{k-1}\left(c_{k-1,i}\tfrac{1}{4i+4}+d_{k-1,i}\tfrac{1}{4i+5}\right)\label{eq:ck0_var1}\\
d_{k,0} & = & \tfrac{1}{6}\sum_{i=0}^{k-1}\left(c_{k-1,i}\tfrac{1}{4i+3}+d_{k-1,i}\tfrac{1}{4i+4}\right)\label{eq:dk0_var1}
\end{eqnarray}
and
\begin{equation}
\left.\begin{array}{ccc}
c_{k,i} & = & -c_{k-1,i-1}\frac{(4i-2)!}{(4i+2)!}\\
d_{k,i} & = & -d_{k-1,i-1}\frac{(4i-1)!}{(4i+3)!}
\end{array}\quad\right\} \quad1\leq i\leq k\,.\label{eq:cki_dki_var1}
\end{equation}
It can be observed that, according to (\ref{eq:ck0_var1}) and (\ref{eq:dk0_var1}),
the coefficients $c_{k,0}$ and $d_{k,0}$ of the powers $z^{2}$
and $z^{3}$ in (\ref{eq:Ansatz_alpha2k}) depend on all coefficients
of (\ref{eq:Ansatz_alpha2k-1}) and on $\beta_{2k}$. In contrast,
the coefficients of the higher powers of $z$ in (\ref{eq:Ansatz_alpha2k})
depend only on one coefficient of (\ref{eq:Ansatz_alpha2k-1}) each
(cf. (\ref{eq:cki_dki_var1})).

Next, it is straightforward to show that the coefficients $c_{k,0}$
and $d_{k,0}$ of (\ref{eq:Ansatz_alpha2k}) can alternatively be
expressed by the sequences $(c_{0,0},c_{1,0},\ldots,c_{k-1,0})$ and
$(d_{0,0},d_{1,0},\ldots,d_{k-1,0})$, i.e. the coefficients of $z^{2}$
and $z^{3}$ of the polynomials $\alpha_{0}(z),\ldots,\alpha_{2(k-1)}(z)$,
as well as $\beta_{2k}$: By a repeated application of (\ref{eq:cki_dki_var1})
we immediately get
\begin{equation}
\left.\begin{array}{ccc}
c_{k,i} & = & (-1)^{i}\frac{2!}{(4i+2)!}c_{k-i,0}\\
d_{k,i} & = & (-1)^{i}\frac{3!}{(4i+3)!}d_{k-i,0}
\end{array}\quad\right\} \quad0\leq i\leq k\,,\label{eq:cki_dki_var2}
\end{equation}
and plugging (\ref{eq:cki_dki_var2}) into (\ref{eq:ck0_var1}) and
(\ref{eq:dk0_var1}) results in
\begin{align}
c_{k,0} & =\tfrac{1}{2}\beta_{2k}-\tfrac{1}{2}\sum_{i=0}^{k-1}(-1)^{i}\left(c_{k-1-i,0}\tfrac{2(4i+3)}{(4i+4)!}+d_{k-1-i,0}\tfrac{6(4i+4)}{(4i+5)!}\right)\label{eq:ck0_var2}\\
d_{k,0} & =\tfrac{1}{6}\sum_{i=0}^{k-1}(-1)^{i}\left(c_{k-1-i,0}\tfrac{2}{(4i+3)!}+d_{k-1-i,0}\tfrac{6}{(4i+4)!}\right)\,.\label{eq:dk0_var2}
\end{align}
Based on (\ref{eq:cki_dki_var2}), (\ref{eq:ck0_var2}), and (\ref{eq:dk0_var2}),
we can now formulate one of the key results of the paper.
\begin{thm}
\label{thm:FormalDiffParam_ck0}The functions $\alpha_{2k}(z)$ and
coefficients $\beta_{2k}$ of the formal differential parametrization
(\ref{eq:Param_even}) are uniquely determined by the sequence $(c_{k,0})_{k\in\mathbb{N}}$.
\end{thm}

\begin{proof}
If we fix a sequence $(c_{k,0})_{k\in\mathbb{N}}$, (\ref{eq:dk0_var2})
allows us to calculate step by step the corresponding sequence $(d_{k,0})_{k\in\mathbb{N}}$,
starting with $d_{0,0}=0$. Subsequently, by (\ref{eq:cki_dki_var2}),
the remaining coefficients of the functions $\alpha_{2k}(z)$ of (\ref{eq:Param_w_even})
can be determined from the sequences $(c_{k,0})_{k\in\mathbb{N}}$
and $(d_{k,0})_{k\in\mathbb{N}}$. Finally, solving (\ref{eq:ck0_var2})
for $\beta_{2k}$ yields the coefficients of (\ref{eq:Param_u_even}).
\end{proof}
Based on the fact that the sequence $(c_{k,0})_{k\in\mathbb{N}}$
can be chosen arbitrarily, the following corollary can be formulated.
\begin{cor}
\label{cor:One_to_one}There is a one-to-one correspondence between
arbitrary sequences $(c_{k,0})_{k\in\mathbb{N}}$ and formal differential
parametrizations (\ref{eq:Param_even}) of the Euler-Bernoulli beam.
\end{cor}

\begin{proof}
By Theorem \ref{thm:FormalDiffParam_ck0}, every sequence $(c_{k,0})_{k\in\mathbb{N}}$
determines a unique formal differential parametrization (\ref{eq:Param_even}).
Conversely, since the sequence $(c_{k,0})_{k\in\mathbb{N}}$ consists
of the coefficients of $z^{2}$ in the functions $\alpha_{2k}(z)$
of (\ref{eq:Param_w_even}), every formal differential parametrization
(\ref{eq:Param_even}) determines a unique sequence $(c_{k,0})_{k\in\mathbb{N}}$.
\end{proof}
As already mentioned in Section \ref{sec:beam}, the functions $\alpha_{2k}(z)$
of (\ref{eq:Param_w_even}) are uniquely determined by the coefficients
$(\beta_{0},\beta_{2},\ldots,\beta_{2k})$ of (\ref{eq:Param_u_even}).
Hence, the formal differential parametrization (\ref{eq:Param_even})
is also uniquely determined by the sequence $(\beta_{2k})_{k\in\mathbb{N}}$.
However, a representation of all formal differential parametrizations
of the Euler-Bernoulli beam using the sequence $(c_{k,0})_{k\in\mathbb{N}}$
as the free design parameter is advantageous. This becomes apparent
when we want to find flat parametrizations, where the series (\ref{eq:Param_w_even})
and (\ref{eq:Param_u_even}) converge for functions $y(t)$ of appropriate
Gevrey order: Since the input $u(t)$ is the bending moment at the
free boundary, which follows from the deflection $w(z,t)$ as $u(t)=\partial_{z}^{2}w(1,t)$,
the convergence of the series (\ref{eq:Param_w_even}) implies the
convergence of the series (\ref{eq:Param_u_even}), under the assumption
that $w(z,t)$ is twice differentiable with respect to $z$ at $z=1$.
Thus, it is sufficient to check the convergence of the series representation
(\ref{eq:Param_w_even}) of the deflection $w(z,t)$. In (\ref{eq:Param_w_even}),
the elements of the sequence $(c_{k,0})_{k\in\mathbb{N}}$ appear
as coefficients of the polynomials $\alpha_{2k}(z)$. This facilitates
the construction of differential parametrizations, as will be illustrated
by means of the examples in Section \ref{sec:Examples}.

In the following, we derive explicit expressions for the sequences
$(d_{k,0})_{k\in\mathbb{N}}$ and $(\beta_{2k})_{k\in\mathbb{N}}$
in terms of the sequence $(c_{k,0})_{k\in\mathbb{N}}$.
\begin{lem}
The sequence $(d_{k,0})_{k\in\mathbb{N}}$ is generated by the discrete
convolution
\begin{equation}
d_{k,0}=\sum_{i=0}^{k}\eta_{k-i}c_{i,0}\label{eq:Convolution_dk0}
\end{equation}
of the sequence $(c_{k,0})_{k\in\mathbb{N}}$ with the sequence $(\eta_{k})_{k\in\mathbb{N}}$,
which is defined recursively by $\eta_{0}=0$ and
\begin{equation}
\eta_{k}=-\tfrac{(-1)^{k}}{3(4k-1)!}-\sum_{i=1}^{k}\eta_{k-i}\tfrac{(-1)^{i}}{(4i)!}\,,\quad k\geq1\,.\label{eq:etak}
\end{equation}
\end{lem}

\begin{proof}
First, it should be noted that $d_{k,0}$ could be calculated from
(\ref{eq:dk0_var2}) by eliminating successively $d_{k-1,0},d_{k-2,0},\ldots$
with shifted versions of (\ref{eq:dk0_var2}). This shows that $d_{k,0}$
is a linear combination of the coefficients of the sequence $(c_{0,0},c_{1,0},\ldots,c_{k,0})$,
justifying an ansatz of the form (\ref{eq:Convolution_dk0}). Due
to $d_{0,0}=0$, from (\ref{eq:Convolution_dk0}) we immediately get
$\eta_{0}=0$. In order to determine $\eta_{k}$, $k\geq1$, we plug
the ansatz (\ref{eq:Convolution_dk0}) into the relation (\ref{eq:dk0_var2}).
After some simplifications this yields the equation
\begin{equation}
\sum_{i=0}^{k}\eta_{k-i}c_{i,0}=\tfrac{1}{3}\sum_{n=0}^{k-1}(-1)^{k-1-n}\tfrac{1}{(4(k-n)-1)!}c_{n,0}+\sum_{n=0}^{k-1}(-1)^{k-1-n}\sum_{j=0}^{n}\eta_{n-j}\tfrac{1}{(4(k-n))!}c_{j,0}\,,\label{eq:Equation_etak}
\end{equation}
which must hold for every choice of the sequence $(c_{0,0},c_{1,0},\ldots,c_{k,0})$.
By the special choice $(1,0,\ldots,0)$, the relation (\ref{eq:Equation_etak})
can be simplified to (\ref{eq:etak}).
\end{proof}
In a similar way, the sequence $(\beta_{2k})_{k\in\mathbb{N}}$ can
be calculated from the sequence $(c_{k,0})_{k\in\mathbb{N}}$.
\begin{lem}
The sequence $(\beta_{2k})_{k\in\mathbb{N}}$ is generated by the
discrete convolution
\begin{equation}
\beta_{2k}=\sum_{i=0}^{k}\mu_{k-i}c_{i,0}\label{eq:Convolution_beta2k}
\end{equation}
of the sequence $(c_{k,0})_{k\in\mathbb{N}}$ with the sequence $(\mu_{k})_{k\in\mathbb{N}}$,
which is defined recursively by $\mu_{0}=2$ and
\begin{equation}
\mu_{k}=\tfrac{4^{k}}{(4k)!}-\sum_{i=1}^{k}\mu_{k-i}\tfrac{(-1)^{i}}{(4i)!}\,,\quad k\geq1\,.\label{eq:muk}
\end{equation}
\end{lem}

\begin{proof}
In principle, $\beta_{2k}$ can be calculated by solving (\ref{eq:ck0_var2})
for $\beta_{2k}$ and replacing $d_{k-1,0},d_{k-2,0},\ldots$ successively
with shifted versions of (\ref{eq:dk0_var2}). This shows that $\beta_{2k}$
is a linear combination of the coefficients of the sequence $(c_{0,0},c_{1,0},\ldots,c_{k,0})$,
justifying an ansatz of the form (\ref{eq:Convolution_beta2k}). A
comparison of (\ref{eq:alpha_0}) and (\ref{eq:Ansatz_alpha2k}) shows
that $c_{0,0}=\frac{1}{2}\beta_{0}$, and from (\ref{eq:Convolution_beta2k})
with $k=0$ we immediately get $\mu_{0}=2$. In order to determine
$\mu_{k}$, $k\geq1$, we solve (\ref{eq:ck0_var2}) for $\beta_{2k}$
and insert the ansatz (\ref{eq:Convolution_beta2k}), which yields
\begin{equation}
\sum_{i=0}^{k}\mu_{k-i}c_{i,0}=2c_{k,0}+\sum_{i=0}^{k-1}(-1)^{i}\left(c_{k-1-i,0}\tfrac{2(4i+3)}{(4i+4)!}+d_{k-1-i,0}\tfrac{6(4i+4)}{(4i+5)!}\right)\,.\label{eq:Equation_beta2k}
\end{equation}
To get rid of the coefficients $c_{i,0}$ and $d_{i,0}$ in (\ref{eq:Equation_beta2k}),
we can use the fact that the particular sequence $c_{k,0}=\frac{(-1)^{k}}{(4k)!}$
determines the sequence $d_{k,0}=-\frac{4k(-1)^{k}}{3(4k)!}$. This
follows rather easily from (\ref{eq:Convolution_dk0}) and (\ref{eq:etak})
and will be shown in Section \ref{subsec:Spezialfall_Parametrierung_Rudolph}.
After inserting these particular sequences, (\ref{eq:Equation_beta2k})
can be simplified to (\ref{eq:muk}).
\end{proof}
A numerical evaluation of the recursively defined sequences (\ref{eq:etak})
and (\ref{eq:muk}) strongly suggests, that $\eta_{k}$ and $\mu_{k}$
are defined by
\[
\eta_{k}=\tfrac{4^{k+1}(1-16^{k})}{6(4k)!}B_{4k}\,,\quad k\geq0
\]
and
\[
\mu_{k}=\tfrac{2}{4^{k}(4k)!}\sum_{i=0}^{2k}(-1)^{i}E_{2i}\tbinom{4k}{2i}\,,\quad k\geq0\,,
\]
where $B_{i}$ denotes the Bernoulli numbers and $E_{i}$ the Euler
numbers. This explicit representation may be advantageous for a convergence
analysis of the parametrizations (\ref{eq:Param_w_even}) and (\ref{eq:Param_u_even}).

\section{\label{sec:Examples}Two Notable Differential Parametrizations}

In the previous section, we have shown that the sequence $(c_{k,0})_{k\in\mathbb{N}}$
can be used as a design parameter for the formal differential parametrizations
(\ref{eq:Param_even}) of the Euler-Bernoulli beam. Based on the representations
(\ref{eq:Convolution_dk0}), (\ref{eq:etak}) and (\ref{eq:Convolution_beta2k}),
(\ref{eq:muk}) of the sequences $(d_{k,0})_{k\in\mathbb{N}}$ and
$(\beta_{2k})_{k\in\mathbb{N}}$, we show that there is a natural
choice for the sequence $(c_{k,0})_{k\in\mathbb{N}}$ that leads to
a flat parametrization, which is well-known from the literature (see
e.g. \cite{Rudolph_Shaker2003}). Subsequently, we discuss a special
formal differential parametrization where the parametrizing output
$y(t)$ allows a physical interpretation.

\subsection{\label{subsec:Spezialfall_Parametrierung_Rudolph}A Natural Choice}

Consider the sum (\ref{eq:Convolution_dk0}). If we split off the
term with $i=0$ and substitute (\ref{eq:etak}) for $\eta_{k}$,
we get
\begin{equation}
d_{k,0}=-\tfrac{(-1)^{k}}{3(4k-1)!}c_{0,0}+\sum_{i=1}^{k}\eta_{k-i}\left(c_{i,0}-\tfrac{(-1)^{i}}{(4i)!}c_{0,0}\right)\,.\label{eq:Convolution_dk0_rewritten}
\end{equation}
With the special choice
\begin{equation}
c_{k,0}=\tfrac{(-1)^{k}}{(4k)!}c_{0,0}\,,\quad k\geq0\label{eq:Obvious_Choice_ck0}
\end{equation}
for the sequence $(c_{k,0})_{k\in\mathbb{N}}$, the sum in (\ref{eq:Convolution_dk0_rewritten})
vanishes, and therefore (\ref{eq:Convolution_dk0_rewritten}) simplifies
to
\[
d_{k,0}=-\tfrac{(-1)^{k}}{3(4k-1)!}c_{0,0}\,,\quad k\geq1\,.
\]
Expanding with $4k$ finally yields the relation
\begin{equation}
d_{k,0}=-\tfrac{4k(-1)^{k}}{3(4k)!}c_{0,0}\,,\quad k\geq0\,,\label{eq:Result_dk0}
\end{equation}
which also includes the case $k=0$ with $d_{0,0}=0$.

Analogously, splitting off the term with $i=0$ in the sum (\ref{eq:Convolution_beta2k})
and substituting (\ref{eq:muk}) for $\mu_{k}$ yields
\begin{equation}
\beta_{2k}=\tfrac{4^{k}}{(4k)!}c_{0,0}+\sum_{i=1}^{k}\mu_{k-i}\left(c_{i,0}-\tfrac{(-1)^{i}}{(4i)!}c_{0,0}\right)\,.\label{eq:Convolution_beta2k_rewritten}
\end{equation}
By the choice (\ref{eq:Obvious_Choice_ck0}) for the sequence $(c_{k,0})_{k\in\mathbb{N}}$,
it is evident that the sum vanishes again, and (\ref{eq:Convolution_beta2k_rewritten})
simplifies to
\[
\beta_{2k}=\tfrac{4^{k}}{(4k)!}c_{0,0}\,,\quad k\geq1\,.
\]
For $k=0$ we have $\beta_{0}=2c_{0,0}$, and therefore the complete
sequence is given by
\begin{equation}
\beta_{2k}=\begin{cases}
2c_{0,0}\,, & k=0\\
\tfrac{4^{k}}{(4k)!}c_{0,0}\,, & k\geq1\,.
\end{cases}\label{eq:beta2k_Rudolph}
\end{equation}
Finally, plugging (\ref{eq:Obvious_Choice_ck0}) and (\ref{eq:Result_dk0})
into (\ref{eq:cki_dki_var2}) results in the complete set of coefficients
\begin{align}
c_{k,i} & =\tfrac{4(-1)^{k}}{(4i+2)!(4(k-i))!}c_{0,0}\,, &  & 0\leq i\leq k\,,\,k\geq0\label{eq:cki_Rudolph}\\
d_{k,i} & =-\tfrac{16(-1)^{k}(k-i)}{(4i+3)!(4(k-i))!}c_{0,0}\,, &  & 0\leq i\leq k\,,\,k\geq0\label{eq:dki_Rudolph}
\end{align}
of the polynomials (\ref{eq:Ansatz_alpha2k}). If we set the scaling
factor in (\ref{eq:Obvious_Choice_ck0}) to $c_{0,0}=2$, we get the
same differential parametrization\begin{subequations}\label{eq:Param_Rudolph}
\begin{eqnarray}
w(z,t) & = & \sum_{k=0}^{\infty}4(-1)^{k}\sum_{i=0}^{k}\left(\tfrac{1}{(4i+2)!(4(k-i))!}z^{4i+2}-\tfrac{4(k-i)}{(4i+3)!(4(k-i))!}z^{4i+3}\right)y^{(2k)}(t)\label{eq:Param_w_Rudolph}\\
u(t) & = & 4y(t)+\sum_{k=1}^{\infty}2\tfrac{4^{k}}{(4k)!}y^{(2k)}(t)\label{eq:Param_u_Rudolph}
\end{eqnarray}
\end{subequations}that was derived, e.g., in \cite{Rudolph_Shaker2003}
by means of operational calculus.\footnote{In contrast to our double sum representation (\ref{eq:Param_w_Rudolph}),
in \cite{Rudolph_Shaker2003}, the parametrization of the deflection
$w(z,t)$ is expressed by a single sum of real and imaginary parts
of powers of complex numbers. Also, the input used in \cite{Rudolph_Shaker2003}
has the opposite sign as compared to our input $u(t)$.} As shown in \cite{Rudolph_Shaker2003}, the series (\ref{eq:Param_w_Rudolph})
and (\ref{eq:Param_u_Rudolph}) converge for all trajectories $y(t)$
of Gevrey class $\gamma<2$. Hence, this parametrizing output $y(t)$
is a flat output.

\subsection{Formal Differential Parametrization by the Bending Moment at the
Clamped Boundary}

In motion planning, one is often interested in a physical interpretation
of the parametrizing output $y(t)$ in terms of a boundary value of
the system. For the flat output $y(t)$ of the differential parametrization
(\ref{eq:Param_Rudolph}), no such physical interpretation is known.
In contrast, for the fully-actuated Euler-Bernoulli beam, with both
the bending moment and the shear force at the free end as inputs,
it is not difficult to show that the bending moment and the shear
force at the clamped end form a flat output. For our underactuated
configuration (\ref{eq:PDE_RB_Beam}), we show that there exists at
least one (useful) formal differential parametrization (\ref{eq:Param_even})
where the parametrizing output $y(t)$ allows a physical interpretation.

First, let us restate the formal differential parametrization of the
deflection $w(z,t)$ by plugging (\ref{eq:Ansatz_alpha2k}) into (\ref{eq:Param_w_even}):
\[
w(z,t)=\sum_{k=0}^{\infty}\sum_{i=0}^{k}\left(c_{k,i}z^{4i+2}+d_{k,i}z^{4i+3}\right)y^{(2k)}(t)\,.
\]
Evaluating its second spatial derivative at $z=0$ yields the parametrization
of the bending moment
\begin{equation}
\partial_{z}^{2}w(0,t)=\sum_{k=0}^{\infty}2c_{k,0}y^{(2k)}(t)\label{eq:Bending_Moment_Clamped_Boundary}
\end{equation}
at the clamped boundary. If we choose the sequence $(c_{k,0})_{k\in\mathbb{N}}$
as
\begin{equation}
(c_{0,0},c_{1,0},c_{2,0},\ldots)=(\tfrac{1}{2},0,0,\ldots)\,,\label{eq:ck0_impulse}
\end{equation}
only the first term of the sum is left, and (\ref{eq:Bending_Moment_Clamped_Boundary})
simplifies to
\begin{equation}
y(t)=\partial_{z}^{2}w(0,t)\,.\label{eq:Interp_y}
\end{equation}
Thus, in this case, the function $y(t)$ is simply the bending moment
at the clamped boundary. By the choice (\ref{eq:ck0_impulse}) the
discrete convolutions (\ref{eq:Convolution_dk0}) and (\ref{eq:Convolution_beta2k})
simplify to $d_{k,0}=\tfrac{1}{2}\eta_{k}\,,\,k\geq0$ and $\beta_{2k}=\tfrac{1}{2}\mu_{k}\,,\,k\geq0$.
Then with (\ref{eq:cki_dki_var2}) we get the formal differential
parametrization\begin{subequations}\label{eq:Param_Bending_Moment}
\begin{eqnarray}
w(z,t) & = & \sum_{k=0}^{\infty}\left((-1)^{k}\tfrac{1}{(4k+2)!}z^{4k+2}+\sum_{i=0}^{k-1}(-1)^{i}\tfrac{3}{(4i+3)!}\eta_{k-i}z^{4i+3}\right)y^{(2k)}(t)\label{eq:Param_w_Bending_Moment}\\
u(t) & = & \sum_{k=0}^{\infty}\mu_{k}y^{(2k)}(t)\,.\label{eq:Param_u_Bending_Moment}
\end{eqnarray}
\end{subequations}However, since the sequence $(\mu_{k})_{k\in\mathbb{N}}$
does not go to zero fast enough, (\ref{eq:Param_u_Bending_Moment})
cannot converge for any non-analytic function $y(t)$ of Gevrey order
$\gamma>1$. Hence, neither does (\ref{eq:Param_w_Bending_Moment}).
A numerical evaluation of the elements $\mu_{k}$ reveals that $\frac{\mu_{k+1}}{\mu_{k}}\approx\frac{1}{24}$,
for $k\ge2$. Since $24^{k}<(2k)!$ for large $k$, the bound $\sup_{t\in\mathbb{R}^{+}}\left|y^{(2k)}(t)\right|$
guaranteed by (\ref{eq:Gevrey_class}) grows much faster than the
coefficients $\mu_{k}$ converge to zero. Consequently, the product
$\mu_{k}y^{(2k)}(t)$ diverges and the parametrizing output $y(t)$
is not a flat one.

However, a flatness-based transition between two equilibria is still
possible using the divergent series (\ref{eq:Param_u_Bending_Moment}).
In \cite{LarocheMartin_IJRNC2000}, simulation results for a heat
conduction problem showed that a least term summation allows feedforward
control even based on divergent series. More precisely, the divergent
series discussed in \cite{LarocheMartin_IJRNC2000} first converges
very fast and then diverges very fast. A similar effect can be observed
for our beam parametrization (\ref{eq:Param_u_Bending_Moment}), for
suitable trajectories $y(t)$. Now, the idea of the least term summation
in \cite{LarocheMartin_IJRNC2000} is to take into account only the
convergent part for each time $t$. Hence, instead of the series (\ref{eq:Param_u_Bending_Moment}),
the feedforward control $u(t)$ is calculated by 
\begin{equation}
u(t)=\sum_{k=0}^{n_{t}}\mu_{k}y^{(2k)}(t)\,,\label{eq:u_LeastTermSummation}
\end{equation}
with $n_{t}$ defined for each $t$ as the smallest integer greater
than $1$ that meets
\[
\left|\mu_{n_{t}+1}y^{(2(n_{t}+1))}(t)\right|>\left|\mu_{n_{t}}y^{(2n_{t})}(t)\right|\,.
\]

Simulation studies were performed in order to illustrate the usefulness
of this approach. Here, we defined the desired trajectory for $y(t)$
based on a function
\[
\Phi_{\sigma}(t)=\begin{cases}
0 & t\leq0\\
\tfrac{\int_{0}^{t}\phi_{\sigma}(\tau)\mathrm{d}\tau}{\int_{0}^{1}\phi_{\sigma}(\tau)\mathrm{d}\tau} & t\in(0,1)\\
1 & t\geq1\,,
\end{cases}
\]
with $\phi_{\sigma}(t)=\exp\left(-1/((1-t)t)^{\sigma}\right)$, scaled
for a transition in $1\,\mathrm{s}$. This function is Gevrey of order
$1+\tfrac{1}{\sigma}$ for $\sigma>0$. In order to transition the
beam between the equilibria $w(z,0)=0$ and $w(z,T)=\tfrac{1}{2}z^{2}$
in finite time $T$, based on the definition of the parametrizing
output $y(t)$ in (\ref{eq:Interp_y}), the corresponding initial
and final value of $y(t)$ are $y(0)=0$ and $y(T)=1$. By setting
$T=5\,\mathrm{s}$ and $\sigma=1.1$, we use the same reference trajectory

\begin{equation}
y(t)=\Phi_{\sigma}(\tfrac{t}{T})\label{eq:Reference_trajectory_y}
\end{equation}
that was used in \cite{HaasRudolph1999} for the flat output in (\ref{eq:Param_Rudolph}).
\begin{rem}
In \cite{HaasRudolph1999}, the transition time is $11.5\,\mathrm{ms}$.
As our normalized beam and the one with physical parameters in \cite{HaasRudolph1999}
are related by a scaling of the spatial variable and a time-scaling
with the factor $\tfrac{1}{2.3}10^{3}$, this corresponds to the transition
time $T=5\,\mathrm{s}$ used in our simulation.
\end{rem}

Fig. \ref{fig:Deflection} and Fig. \ref{fig:y} give simulation results
for the feedforward control (\ref{eq:u_LeastTermSummation}) based
on the reference trajectory (\ref{eq:Reference_trajectory_y}). The
distributed beam deflection in Fig. \ref{fig:Deflection} shows that
the desired transition is achieved. In Fig. \ref{fig:y}, the (simulated)
bending moment at the clamped boundary, i.e. the parametrizing output
$y(t)$, matches the desired reference trajectory (\ref{eq:Reference_trajectory_y}),
apart from some small deviations. Thus, using only the formal differential
parametrization (\ref{eq:Param_Bending_Moment}), we managed to implement
the same transition considered in \cite{HaasRudolph1999} on the basis
of the flat differential parametrization (\ref{eq:Param_Rudolph}).
However, in contrast to the flat output used in \cite{HaasRudolph1999},
our parametrizing output offers a physical interpretation. This is
particularly interesting for applications where the bending moment
must not exceed certain bounds, since we can take account of these
bounds in the design of the desired trajectory $y(t)$ directly.

Nevertheless, it is important to mention that choosing the transition
time $T$ too small means that the series (\ref{eq:Param_u_Bending_Moment})
diverges too soon. Consequently, a least term summation would not
make sense any more. In contrast, increasing the transition time improves
the results, in particular with respect to the small deviations visible
in Fig. \ref{fig:y}. For instance, with $T=10\,\mathrm{s}$ the simulated
bending moment at the clamped boundary almost perfectly matches the
reference (\ref{eq:Reference_trajectory_y}). We also expect that
evaluating the divergent series (\ref{eq:Param_u_Bending_Moment})
with the more sophisticated summation methods used in \cite{WagnerMeurerZeitz2004}
and \cite{WagnerMeurerKugi2008} should further improve the results.

\begin{figure}
\noindent \begin{centering}
\includegraphics[width=0.5\columnwidth]{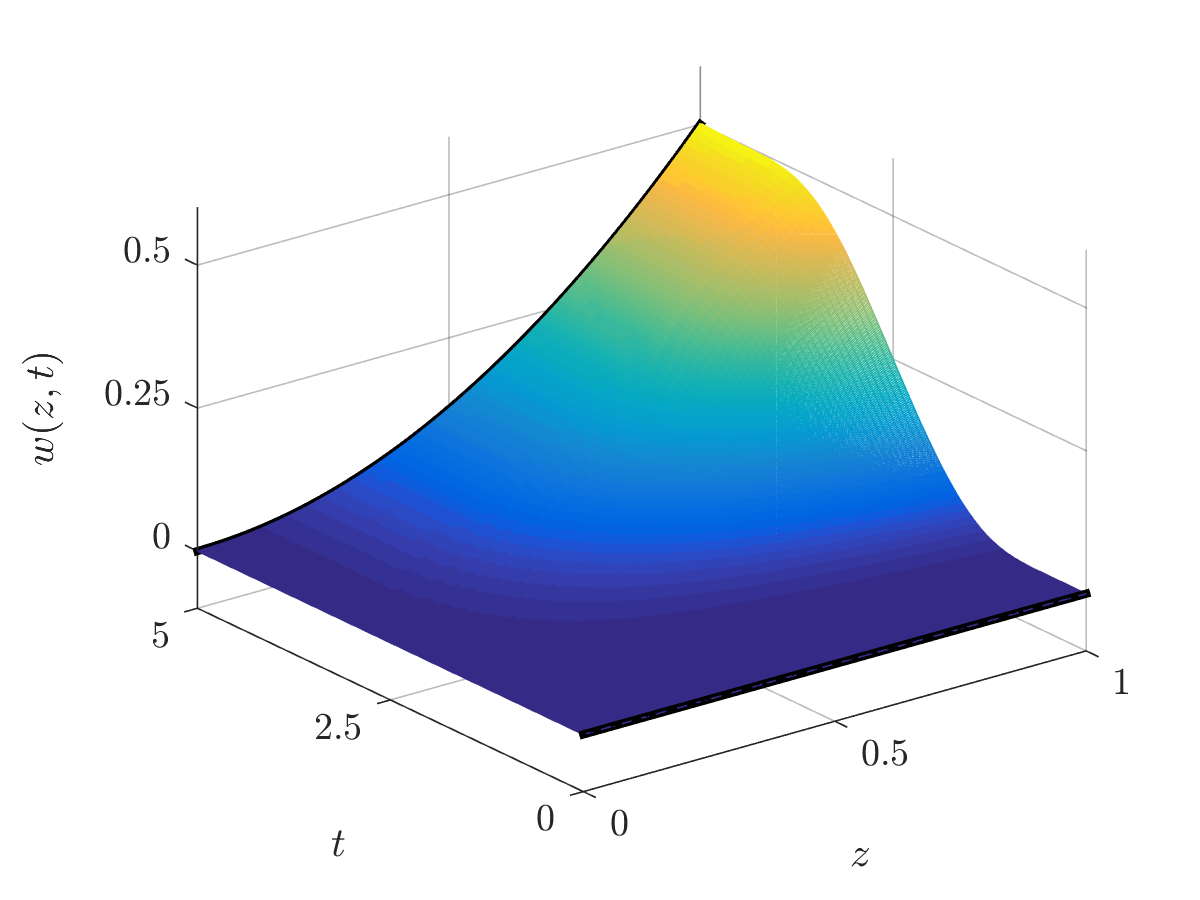}
\par\end{centering}
\caption{\label{fig:Deflection}Deflection $w(z,t)$ of the beam}
\end{figure}
\begin{figure}
\noindent \centering{}\includegraphics[width=0.5\columnwidth]{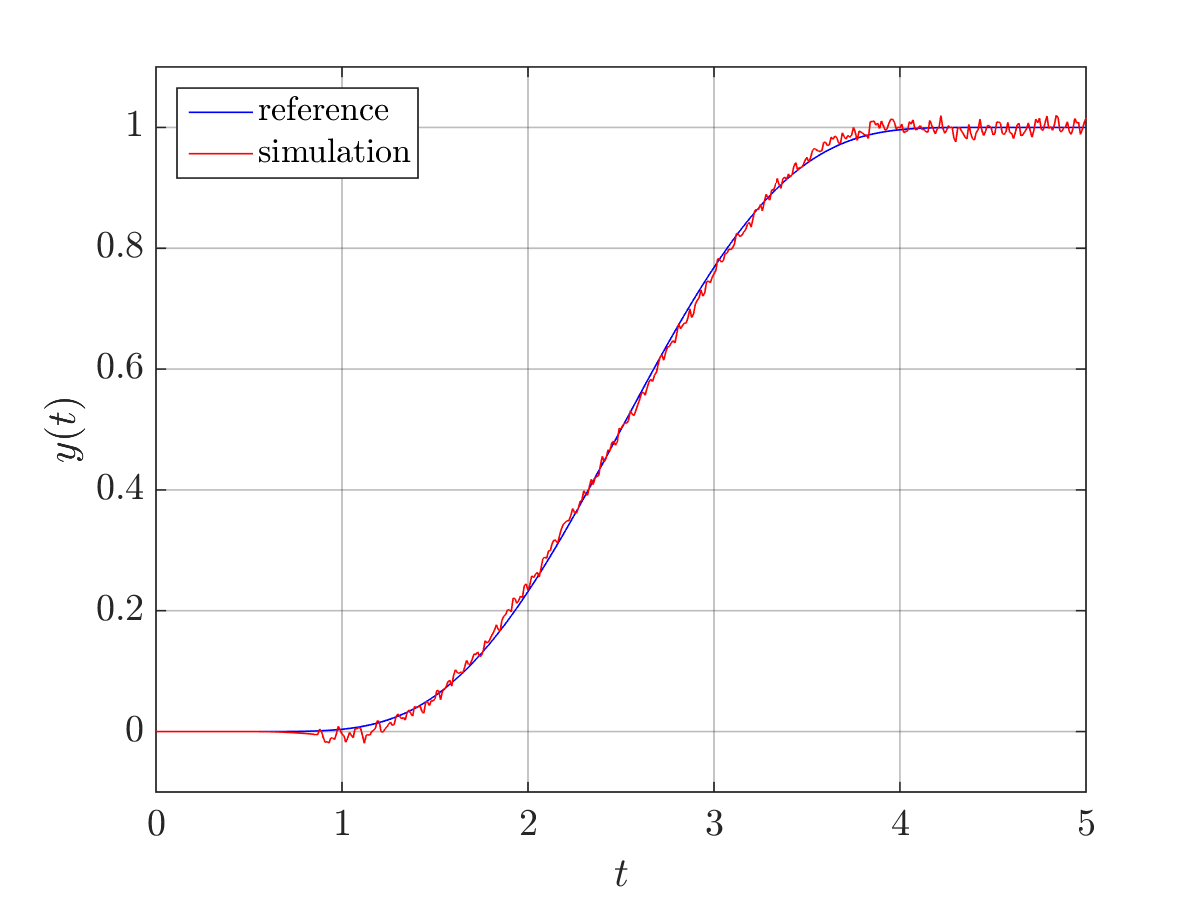}\caption{\label{fig:y}Bending moment $y(t)$ at the clamped boundary}
\end{figure}

\paragraph{Acknowledgements}

This work has been supported by the Austrian Science Fund (FWF) under
grant number P 29964-N32 and the Pro$^2$Future competence center
in the framework of the Austrian COMET-K1 programme under contract
no. 854184.

\bibliographystyle{IEEEtran}
\bibliography{Literature_Bernd}

\end{document}